\documentclass[12pt]{amsart}
\usepackage{amsfonts,color,morefloats,pslatex}
\usepackage{amssymb,amsthm, amsmath,latexsym}

\newtheorem{theorem}{Theorem}
\newtheorem{lemma}[theorem]{Lemma}

\newtheorem{example}{Example}

\def\qi#1 {\fbox {\footnote {\ }}\ \footnotetext { From Qi: {\color{red}#1}}}

\begin{document}
\title[Infinite families of $2$-designs from linear codes]{Infinite families of $2$-designs from a class of non-binary Kasami cyclic codes}

\author{Rong Wang}
\address{College of Mathematics and Statistics, Northwest Normal University, Lanzhou, Gansu 730070, China}
\curraddr{}
\email{rongw113@126.com}
\author{Xiaoni Du}
\address{College of Mathematics and Statistics, Northwest Normal University, Lanzhou, Gansu 730070, China}
\curraddr{}
\email{ymldxn@126.com}

\author{Cuiling Fan}
\address{ School of Mathematics,
 Southwest Jiaotong,
 University, Chengdu, China}
\curraddr{}
\email{ cuilingfan@163.com}

%

\thanks{}
\date{\today}

\maketitle

\begin{abstract}
Combinatorial $t$-designs have been an important research subject for many years, as they have wide applications in coding theory, cryptography, communications and statistics. The interplay between coding theory and $t$-designs has been attracted a lot of attention for both directions. It is well known that a linear code over any finite field can be derived from the incidence matrix of a $t$-design, meanwhile, that the supports of all codewords with a fixed weight in a code also may hold a $t$-design. In this paper, by determining the weight distribution of a class of linear codes derived from non-binary Kasami cyclic codes, we obtain infinite families of $2$-designs from the supports of all codewords with a fixed weight in these codes, and calculate their parameters explicitly.

\end{abstract}

\section{Introduction}
We begin this paper by a brief introduction to $t$-designs.
Let $t,k$ and $v$ be positive integers satisfying  $1\leq t \leq k \leq v.$
Let $\mathcal{P}=\{0, 1, \ldots, v-1\}$, $\mathcal{B}$ be a multi-set of $k$-subsets of $\mathcal{P},$  and the size of $\mathcal{B}$ be denoted by $b$.  If every $t$-subset of $\mathcal{P}$ is contained in exactly $\lambda$ elements of $\mathcal{B},$ then we call the pair $\mathbb{D}=(\mathcal{P}, \mathcal{B})$ a $t$-$(v,k,\lambda)$ design, or simply a {\em
$t$-design}. The elements of $\mathcal{P}$ are called {\em points}, and those of $\mathcal{B}$ are referred to as {\em blocks}. A $t$-design is {\em simple} when $\mathcal{B}$ does not contain repeated blocks, and a $t$-design is called {\em symmetric} if $v=b$ and {\em trivial} if $k=t$ or $k=v$. In general,  researchers pay only attention to simple $t$-designs with $t < k < v$. When $t\geq
2$ and $\lambda=1,$ we call the $t$-$(v, k, \lambda)$ design a  {\em Steiner system}. The following identity is a necessary condition for a $t$-$(v,k,\lambda)$ design.
\begin{equation}\label{condition}
  b {k \choose t} = \lambda {v \choose t}.
\end{equation}

There is an intimate relationship between linear codes and $t$-designs which has been attracted a lot of attention during these decades. On  one hand, a linear code over any {\em finite field} can be derived from the incidence matrix of a $t$-design and much progress has been made (see for example \cite{AB92,DD15,Ton98,Ton07}). On the other hand, linear and nonlinear codes both can give  $t$-designs. There are two standard way to derive $t$-designs from linear codes. One is based on Assums-Matterson Theorem \cite{AM69,AM74} and the other is via the automorphism group of a linear code $\mathcal{C}$ which $\mathcal{C}$ holds $t$-designs if the permutation part of the automorphism group acts $t$-transitively on the code $\mathcal{C}$ \cite{AK92,MS771}.   
Very recently, infinite families of $2$-designs and $3$-designs were obtained from several different classes of linear codes by Ding \cite{Ding182}, and Ding and Li \cite{DL17} by applying the Assums-Matterson Theorem. After that, with the latter method, Du {\em et al.} have derived infinite families of $2$-designs from some different classes of affine-invariant codes \cite{DWTW190,DWTW191,DWF192}. Some other constructions of $t$-designs can be found in \cite{BJL99,CM06,MS771,RR10}. In this paper, following the approach of $2$-transitivity, we will derive some infinitely classes of $2$-designs from a class of linear codes derived from non-binary Kasami cyclic codes.

Now we review some basic definitions of linear codes. Let $p$ be an  prime and $m$ a positive integer.  Let $\mathbb{F}_q$ denote the finite field with $q=p^m$ elements and $\mathbb{F}^*_q=\mathbb{F}_q\backslash\{0\}$. An $[n,k,\delta ]$  {\em linear code} $\mathcal{C}$ over $\mathbb{F}_p$ is a $k$-dimensional subspace of $\mathbb{F}_p^n$ with minimum Hamming distance $\delta $, and is called a {\em cyclic code} if each codeword $(c_0,c_1,\ldots,c_{n-1})\in\mathcal{C}$ implies $(c_{n-1}, c_0, c_1, \ldots,
c_{n-2}) \in \mathcal{C}$. It is well known that any cyclic code $\mathcal{C}$ can be expressed as $\mathcal{C} = \langle g(x) \rangle$, where $g(x)$ is monic and has the least degree. The polynomial  $g(x)$ is called the  {\em generator polynomial} and $h(x)=(x^n-1)/g(x)$ is referred to as the {\em parity-check polynomial} of $\mathcal{C}$. We call the $[n,n-k]$ code, whose generator polynomial is $x^kh(x^{-1})/h(0)$,  the {\em dual} of $\mathcal{C}$, denoted by $\mathcal{C}^{\bot}$.  Furthermore, we define the  {\em extended code} $\overline{\mathcal{C}}$ of $\mathcal{C}$ to be
$$\overline{\mathcal{C}}=\{(c_0, c_1, \ldots, c_n) \in \mathbb{F}_p^{n+1}: (c_0, c_1, \ldots, c_{n-1}) \in \mathcal{C} ~\mathrm{with} ~\sum^n_{i=0}c_i=0\}.
$$

Let $A_i$ be the number of codewords with Hamming weight $i$ in a code $\mathcal{C}$. The {\em weight enumerator} of $\mathcal{C}$ is defined by
$1+A_1z+A_2z^2+\ldots+A_nz^n,$
and the sequence $(1, A_1, \ldots, A_n)$ is called the {\em weight distribution} of the code $\mathcal{C}.$ If the number of nonzero $A_i$'s with $1 \leq i \leq n$ is $w$, then we call $\mathcal{C}$ a {\em $w$-weight code}. If we index the coordinates of a codeword  by $(0, 1, \ldots, n-1)$, then for any codeword $\mathbf{c}=(c_0, c_1, \ldots, c_{n-1}) \in \mathcal{C},$ the {\em support} of $\mathbf{c}$ is defined by
$$Suppt(\mathbf{c})=\{0\leq i \leq n-1: c_i\neq 0\}\subseteq \{0, 1, \ldots, n-1\}.$$
And so  for each $i$ with
$A_i\neq 0$, let $\mathcal{B}_i$ define the set of the supports of all codewords with weight $i$ in a code $\mathcal{C}$. Let  $\mathcal{P}=\{0,1,\ldots, n-1\}$. Then the pair $(\mathcal{P},\mathcal{B}_i)$ may be a $t$-$(n,i,\lambda)$ design for some $\lambda$ with the well known results given in literature \cite{DL17}. In this case, we say that the codewords of weight $i$ in $\mathcal{C}$ hold a $t$-design, which is called a support design of $\mathcal{C}$.

In the rest of  this paper, we always suppose that  $p$ is an odd prime and $m=2s$, $l(l\neq s)$ are positive integers with $0\leq l\leq m-1$, $gcd(s+l, 2l)=d'$ and $gcd(s, l)=d$. Let $q=p^m$ and  $\mathbb{C}$ be the cyclic code with
length $n=q-1$ and parity-check polynomial $h_1(x)h_{p^s+1}(x)h_{p^l+1}(x),$ where $h_1(x),h_{p^s+1}(x)$ and $h_{p^l+1}(x)$ are the minimal polynomials of $\alpha,$ $\alpha^{p^s+1}$ and $\alpha^{p^l+1}$ over $\mathbb{F}_p,$ respectively, and $\alpha$ is a primitive element of $\mathbb{F}_q$. One can see that $\mathbb{C}$ is the non-binary Kasami cyclic code. Luo {\em et al.} \cite{LTW09} determined  the weight distribution of  $\mathbb{C}$ by explicitly computing the exponential sums since  the weight of each codeword can be expressed by certain exponential sums, and which is a general approach to investigate the weight distribution of cyclic code (see \cite{FY05,Vlugt95} and the references therein).

With the definition of $\mathbb{C}$, we deduce the following linear codes ${\overline{{\mathbb{C}}^{\bot}}}^{\bot}$.
\begin{eqnarray}\label{code-1}
{\overline{\mathbb{C}^\bot}}^\bot:=\{ (Tr_s (ax^{p^s+1})&+& Tr_m( bx^{p^l+1}+cx)+h ): \\
& & a\in \mathbb{F}_{p^s}, x, b , c \in \mathbb{F}_q, h\in \mathbb{F}_p \}, \nonumber
\end{eqnarray}
where  $Tr_r(x)= x + x^{p}+\ldots+x^{p^{r-1}}$ is the  absolute trace function from  $\mathbb{F}_{p^r}$ onto $\mathbb{F}_{p}$ for any positive integer $r$.
In this paper, we will  determine the weight distribution of the linear codes ${\overline{{\mathbb{C}}^{\bot}}}^{\bot}$ and obtain infinite families of  $2$-designs from its weight distribution.

The remainder of this paper is organized as follows. In Section \ref{section-2}, we introduce some preliminary results on exponential sums, cyclotomic fields, affine invariant codes and $2$-designs, which will be used in subsequent sections. In Section \ref{theorem}, we describe the weight distribution of ${\overline{{\mathbb{C}}^{\bot}}}^{\bot}$ given in Eq.(\ref{code-1}) and then present infinite families of $2$-designs with the explicit parameters. In Section \ref{section-4}, we first prove the weight distribution of ${\overline{{\mathbb{C}}^{\bot}}}^{\bot}$  by computing certain exponential sums, and then present the proofs of the results of the $2$-designs given in Section \ref{theorem}. Section \ref{section-5} concludes the paper.

\section{Preliminaries}\label{section-2}
In this section, we summarize some basic facts on affine-invariant codes, $2$-designs, exponential sums and cyclotomic fields. We should mention that most of the known results hold for any positive integer $m$ and prime $p$ unless otherwise stated.


\subsection{Affine-invariant codes and 2-designs}

We begin this subsection by the introduction of affine-invariant codes.

The permutation automorphism group of $\mathcal{C}$, denoted by $PAut(\mathcal{C})$, is the set of coordinate permutations that map a code $\mathcal{C}$ to itself. We define the affine group $GA_1(\mathbb{F}_q)$ by the set of all permutations
$$\sigma_{a,b}: x\mapsto  ax+b $$
of $\mathbb{F}_q$, where $a \in \mathbb{F}_q^*$ and  $b \in \mathbb{F}_q.$
An {\em affine-invariant} code is an extended cyclic code $\overline {\mathcal{C}}$ over $\mathbb{F}_p$ such that $GA_1(\mathbb{F}_q)\subseteq PAut(\overline{\mathcal{C}})$ \cite{HP03}.

For any integers $r,s \in \mathcal{P}$, we define $r\preceq s$ if $r_i \leq s_i$ for all $0\leq i\leq m-1$, where $s=\sum^{m-1}_{i=0}s_ip^i,~0\leq s_i\leq p-1$ and  $r=\sum^{m-1}_{i=0}r_ip^i,~0\leq r_i\leq p-1$  are the $p$-adic expansions of $s$ and $r$, respectively. Clearly, we have $r \le s $ if
 $r\preceq s$.

%

Let $C_j$ be the  {\em $p$-cyclotomic coset} modulo $n$ containing $j$, i.e.,
$$C_j=\{jp^i \pmod n: 0 \le i \le \ell_j-1\},$$
where $j$ is any integer with $0\leq j< n$ and   $\ell_j$ is the smallest positive integer such that $j\equiv jp^{\ell_j}\pmod n.$ The smallest integer in $C_j$ is called the {\em coset leader} of $C_j$. Thus the  generator polynomial
$g(x)$ of a code $\mathcal{C}$ can be write as $g(x)=\prod_j\prod_{e\in C_j}(x-\alpha^e)$, where $j$ runs through some coset leaders of the $p$-cyclotomic cosets $C_j$ modulo $n.$  The set   $T=\bigcup_jC_j$ is referred to as the {\em defining set} of $\mathcal{C}$, which is the union of these $p$-cyclotomic cosets.

 Affine-invariant  codes  are a very important class of linear codes to present $t$-designs.  The first item of following lemma given by Kasami, Lin and Peterson \cite{KLP67} is very important since it points out that one can determine whether a given extended primitive cyclic code $\overline{\mathcal{C}}$ is affine-invariant or not by analysing the properties of its defining set. The second item proposed by Ding \cite{Ding18b} shows that one can get  new affine-invariant codes from the known ones.

\begin{lemma}\label{Kasami-Lin-Peterson} Let $\overline{\mathcal{C}}$ be an extended cyclic code of
  length $p^m$ over $\mathbb{F}_p$ with defining set $\overline{T}$. Then

(1) \cite{KLP67}  The code $\overline{\mathcal{C}}$ is affine-invariant if and
only if whenever $s \in   \overline{T}$  then $r \in \overline{T}$ for all $r \in \mathcal{P}$  with $r \preceq s$.

(2) \cite{Ding18b} The dual of an affine-invariant code $\overline{\mathcal{C}}$ over $\mathbb{F}_p$
is also affine-invariant.
\end{lemma}


\begin{theorem}\label{2-design}
  \cite{Ding18b} For each $i$ with $A_i \neq 0$ in an affine-invariant code $\overline{\mathcal{C}}$ with length $p^m$, the supports of the codewords of weight $i$ form a $2$-design.
\end{theorem}

Lemma \ref{Kasami-Lin-Peterson} and Theorem \ref{2-design} are very attractive in the sense that they can be used to examine the existence of $2$-designs. We will employ them later to present the main results of this paper.

The following theorem given by Ding in \cite{Ding18b} presents the relation of all codewords with the same support in a linear code $\mathcal{C}$, which will be used together with Eq.(\ref{condition}) to calculate the parameters of $2$-designs later.

\begin{theorem}\label{design parameter} \cite{Ding18b} Let $\mathcal{C}$ be a linear code over $\mathbb{F}_p$ with minimum weight $\delta$ and length $n$. Let $w$ be the largest integer with $w\leq n$ satisfying
 $${w-\lfloor\frac{w+p-2}{p-1}\rfloor}<\delta.$$
Let $\mathbf{c_1}$ and $\mathbf{c_2}$ be two codewords of weight $i$ with $\delta \leq i\leq w$ and $Suppt(\mathbf{c_1})=Suppt(\mathbf{c_2}).$ Then $\mathbf{c_1}=a\mathbf{c_2}$ for some $a\in \mathbb{F}_p^*$.
\end{theorem}

\subsection{Exponential sums}

Let $\zeta_N=e^{2\pi \sqrt{-1}/N}$ be a primitive $N$-th root of unity for any positive integer $N \ge 2$. An additive character of $\mathbb{F}_q$ is a nonzero function $\chi$ from $\mathbb{F}_q$ to the set of complex numbers of absolute value $1$ such that $\chi(x+y)=\chi(x)\chi(y)$ for any pair $(x,y) \in \mathbb{F}_q^2$. For each $u \in \mathbb{F}_q$, the function
$$\chi_u(v)=\zeta_p^{Tr_m( uv)},~v \in \mathbb{F}_q$$
denotes an additive character of $\mathbb{F}_q$. Since $\chi_0(v)=1$ for all $v \in \mathbb{F}_q$, we call $\chi_0$ the  {\em trivial} additive character of $\mathbb{F}_q$. We call $\chi_1$ the {\em canonical} additive character of $\mathbb{F}_q$ and we have  $\chi_u(x)=\chi_1(ux)$ for all $u\in\mathbb{F}_q$ \cite{LN97}.

Let $\alpha$ be a primitive element of $\mathbb{F}_q$. For each $j=0,1,\ldots,q-2,$ a multiplicative character  of $\mathbb{F}_q$  is defined by the function $\psi_j(\alpha^k)=\zeta_{q-1}^{jk},$ for $k=0,1,\ldots,q-2.$    We denote by $\eta:=\psi_{\frac{q-1}{2}}$ which  is called the {\em quadratic character} of $\mathbb{F}_q.$ Similarly,  we may define the quadratic character  $\eta ^{'}$  of $\mathbb{F}_p.$ For simplicity, we extend the quadratic characters by setting $\eta(0)=0$ and $\eta'(0)=0$.

The {\em Gauss sum} $G(\eta', \chi'_1)$ over $\mathbb{F}_p$ is defined by $G(\eta', \chi'_1)=\sum\limits_{v \in \mathbb{F}_p^*}\eta'(v)\chi'_1(v)=\sum\limits_{v \in \mathbb{F}_p}\eta'(v)\chi'_1(v)$, where $\chi'_1$ is the canonical additive character of $\mathbb{F}_p$. For {\em Gauss sums}, we have the following  basic fact which is essential to examine the weight distribution of the code ${\overline{{\mathbb{C}}^{\bot}}}^{\bot}$ defined in Eq.(\ref{code-1}).

\begin{lemma}\label{Gauss} \cite{LN97} With the notation above, we have
$$G(\eta', \chi'_1)={\sqrt{(-1)}^{(\frac{p-1}{2})^2}}\sqrt{p}=\sqrt{p^*},$$
where $p^*=(-1)^{\frac{p-1}{2}}p$.
\end{lemma}

To determine the weight distribution  of the code ${\overline{{\mathbb{C}}^{\bot}}}^{\bot}$ defined in Eq.(\ref{code-1}), we introduce the following function
\begin{equation}\label{S(a.b.c)}
S(a,b,c)=\sum\limits_{x \in \mathbb{F}_q}\zeta_p^{Tr_s(ax^{p^s+1})+Tr_m( bx^{p^l+1}+cx)},\qquad a\in \mathbb{F}_{p^s}, b, c \in \mathbb{F}_q.
\end{equation}

For the the value distribution of $S(a,b,c)$, we introduce the earlier results  presented by Luo {\em et al.} \cite{LTW09} in the following.
\begin{lemma}\label{value distribution} \cite{LTW09}
For $m \geq 4$, gcd$(s+l, 2l)=d'$, gcd$(s, l)=d$, $\varepsilon=\pm1$ and $j\in \mathbb{F}_p^*,$ the value distribution of the $\{S(a,b,c) : a\in \mathbb{F}_{p^s}, b, c \in \mathbb{F}_q\}$ is given in Table \ref{1} when $d'=d$ is odd, in Table \ref{2} when $d'=d$ is even and in Table \ref{3} when $d'=2d$, respectively (see Tables~\ref{1}, ~\ref{2} and~\ref{3} in Appendix I).
\end{lemma}


\subsection{Cyclotomic fields}
We  state the following basic facts on Galois group of cyclotomic fields $\mathbb{Q}(\zeta_p)$ since  $S(a,b,c)$ is element in $\mathbb{Q}(\zeta_p)$.

\begin{lemma}\label{Cyclotomic fields} \cite{IR90}
  Let $\mathbb{Z}$ be the rational integer ring and $\mathbb{Q}$ be the rational field.

$(1)$ The ring of integers in $K=\mathbb{Q}(\zeta_p)$ is $\mathcal{O}_k=\mathbb{Z}[\zeta_p]$ and $\{\zeta^i_{p} : 1 \leq i \leq p-1\}$ is an integral basis of $\mathcal{O}_k.$

$(2)$ The filed extension $K/\mathbb{Q}$ is Galois of degree $p-1$ and the Galois group $Gal(K/\mathbb{Q})=\{\sigma_y : y \in (\mathbb{Z}/p\mathbb{Z})^*\},$ where the automorphism $\sigma_y$ of $K$ is defined by $\sigma_y(\zeta_p)=\zeta^y_{p}.$

$(3)$ $K$ has a unique quadratic subfield $L=\mathbb{Q}(\sqrt{p^*}).$ For $1 \leq y \leq p-1, \sigma_y(\sqrt{p^*})
=\eta'(y)\sqrt{p^*}.$ Therefore, the Galois group $Gal(L/\mathbb{Q})$ is $\{1,\sigma_\gamma\},$ where $\gamma$ is any quadratic nonresidue in $\mathbb{F}_p.$
\end{lemma}

\section{Weight distributions of the linear codes and their 2-designs}\label{theorem}
In this section, we will describe  the  weight distribution  of the code ${\overline{{\mathbb{C}}^{\bot}}}^{\bot}$ presented in Eq.(\ref{code-1})  and then obtain infinite families of $2$-designs from the weight distribution of ${\overline{{\mathbb{C}}^{\bot}}}^{\bot}$. We first  present the weight distribution of the code ${\overline{{\mathbb{C}}^{\bot}}}^{\bot}$ in the following theorem.

\begin{theorem}\label{weight1}
For $m\geq 4$, the weight distribution of the codes ${\overline{{\mathbb{C}}^{\bot}}}^{\bot}$ over $\mathbb{F}_p$ with length $p^m$ and $dim({\overline{{\mathbb{C}}^{\bot}}}^{\bot})=\frac{5m}{2}+1$ is given in Table \ref{4} when $d'=d$ is odd, in Table \ref{5} when $d'=d$ is even and in Table \ref{6} when $d'=2d$, respectively.
\begin{table}
\begin{center}
\caption{The weight distribution of ${\overline{{\mathbb{C}}^{\bot}}}^{\bot}$ when $d'=d$ is odd}\label{4}
\begin{tabular}{ll}
\hline\noalign{\smallskip}
Weight  &  Multiplicity   \\
\noalign{\smallskip}
\hline\noalign{\smallskip}
$0$  &  $1$ \\
$(p-1)(p^{m-1}-p^{s-1})$ & $ \frac{1}{2}p^{m+d}(p^s+1)(p^m-1)/(p^d+1)  $     \\
$ p^{m-1}(p-1)+p^{s-1}$  &  $ \frac{1}{2}p^{m+d}(p-1)(p^s+1)(p^m-1)/(p^d+1)$     \\
$(p-1)(p^{m-1}+p^{s-1})$  &  $ \frac{p^{m+d}(p^m-2p^{m-d}+1)(p^s-1)}{2(p^d-1)}$     \\
$ p^{m-1}(p-1)-p^{s-1}$  &  $ \frac{p^{m+d}(p-1)(p^m-2p^{m-d}+1)(p^s-1)}{2(p^d-1)}$     \\
$ p^{m-1}(p-1)\pm (-1)^{\frac{p-1}{2}}p^{s+\frac{d-1}{2}}$  &  $\frac{1}{2}p^{3s-2d}(p-1)(p^m-1)$     \\
$ p^{s+d-1}(p-1)(p^{s-d}+1)$  &  $ p^{m-2d}(p^{s-d}-1)(p^m-1)/(p^{2d}-1)$     \\
$ p^{m-1}(p-1)-p^{s+d-1}$  &  $ \frac{p^{m-2d}(p-1)(p^{s-d}-1)(p^m-1)}{(p^{2d}-1)}$     \\
$p^{m-1}(p-1)$  &  $ p(p^{3s-d}-p^{3s-2d}+p^{3s-2d-1}+p^{3s-3d}$\\ &$-p^{m-2d}+1)(p^m-1) $    \\
$ p^m$  &  $ p-1$     \\
\noalign{\smallskip}
\hline
\end{tabular}
\end{center}
\end{table}

\begin{table}
\begin{center}
\caption{The weight distribution of ${\overline{{\mathbb{C}}^{\bot}}}^{\bot}$ when $d'=d$ is even}\label{5}
\begin{tabular}{ll}
\hline\noalign{\smallskip}
Weight  &  Multiplicity   \\
\noalign{\smallskip}
\hline\noalign{\smallskip}
$0$  &  $1$ \\
$p^{s-1}(p-1)(p^s-1) $ & $ \frac{1}{2}p^{m+d}(p^s+1)(p^m-1)/(p^d+1)  $     \\
$ p^{s-1}(p^{s+1}-p^s+1)$  &  $ \frac{1}{2}p^{m+d}(p-1)(p^s+1)(p^m-1)/(p^d+1)$     \\
$p^{s-1}(p-1)(p^s+1)$  &  $ \frac{p^{s+d-1}(p^m-2p^{m-d}+1)(p^s+1)(p^{s+1}-2p+2)}{2(p^d-1)}$     \\
$  p^{s-1}(p^{s+1}-p^s-1)$  &  $ \frac{p^{s+d-1}(p-1)(p^m-2p^{m-d}+1)(p^s+1)(p^{s+1}-2p+2)}{2(p^d-1)}$    \\
$ p^{s+\frac{d}{2}-1}(p-1)(p^{s-\frac{d}{2}} \pm 1)$  &  $\frac{1}{2}p^{3s-2d}(p^m-1)$     \\
$p^{s+\frac{d}{2}-1}(p^{s-\frac{d}{2}+1}-p^{s-\frac{d}{2}}\pm 1)$  &  $\frac{1}{2}p^{3s-2d}(p-1)(p^m-1)$     \\
$p^{s+d-1}(p-1)(p^{s-d}+1)$  &  $p^{m-2d}(p^{s-d}-1)(p^m-1)/(p^{2d}-1)$     \\
$ p^{s+d-1}(p^{s-d+1}-p^{s-d}-1)$  &  $ \frac{p^{m-2d}(p-1)(p^{s-d}-1)(p^m-1)}{(p^{2d}-1)}$     \\
$p^{m-1}(p-1)$  &  $ p(p^{3s-d}-p^{3s-2d}+p^{3s-3d}-p^{m-2d}$\\ &$+1)(p^m-1) $    \\
$ p^m$  &  $ p-1$     \\
\noalign{\smallskip}
\hline
\end{tabular}
\end{center}
\end{table}

\begin{table}
\begin{center}
\caption{The weight distribution of ${\overline{{\mathbb{C}}^{\bot}}}^{\bot}$ when $d'=2d$}\label{6}
\begin{tabular}{ll}
\hline\noalign{\smallskip}
Weight  &  Multiplicity   \\
\noalign{\smallskip}
\hline\noalign{\smallskip}
$0$  &  $1$ \\
$p^{s-1}(p-1)(p^s+1) $ & $\frac{p^{m+3d}(p^m-p^{m-2d}-p^{m-3d}+p^s-p^{s-d}+1)(p^s-1)}{(p^d+1)(p^{2d}-1)}  $     \\
$ p^{s-1}(p^{s+1}-p^s-1)$  &  $ \frac{p^{m+3d}(p-1)(p^m-p^{m-2d}-p^{m-3d}+p^s-p^{s-d}+1)(p^s-1)}{(p^d+1)(p^{2d}-1)} $     \\
$p^{s+d-1}(p-1)(p^{s-d}-1)$  &  $ \frac{p^{m-d}(p^s+p^{s-d}+p^{s-2d}+1)(p^m-1)}{(p^d+1)^2}$     \\
$  p^{s+d-1}(p^{s-d+1}-p^{s-d}+1)$  &  $\frac{p^{m-d}(p-1)(p^s+p^{s-d}+p^{s-2d}+1)(p^m-1)}{(p^d+1)^2}$    \\
$p^{s+2d-1}(p-1)(p^{s-2d}+1)$  &  $\frac{p^{m-4d}(p^{s-d}-1)(p^m-1)}{(p^d+1)(p^{2d}-1)}$     \\
$p^{s+2d-1}(p^{s-2d+1}-p^{s-2d}-1)$  &  $\frac{p^{m-4d}(p-1)(p^{s-d}-1)(p^m-1)}{(p^d+1)(p^{2d}-1)}$     \\
$p^{m-1}(p-1)$  &  $ p(p^{3s-d}-p^{3s-2d}+p^{3s-3d}-p^{3s-4d}$\\ &$+p^{3s-5d}+p^{m-d}-2p^{m-2d}+p^{m-3d}$\\ &$-p^{m-4d}+1)(p^m-1) $    \\
$ p^m$  &  $ p-1$     \\
\noalign{\smallskip}
\hline
\end{tabular}
\end{center}
\end{table}
\end{theorem}

One can see that the code is at most ten-weight if $d'=d$ is odd, twelve-weight if $d'=d$ is even and eight-weight if $d'=2d$. 

\begin{theorem}\label{$2-$design-1}
Let $m\geq 4$ be a positive integer. Then the supports of the all codewords of weight $i>0$ in ${\overline{{\mathbb{C}}^{\bot}}}^{\bot}$ form a $2$-design, provided that $A_i\neq0.$
\end{theorem}
%
 We can deduce   the parameters of the 2-designs derived from ${\overline{{\mathbb{C}}^{\bot}}}^{\bot}$ in the following theorem.
\begin{theorem}\label{parameter-1}
Let $\mathcal{B}$ be the set of the supports of the codewords of ${\overline{{\mathbb{C}}^{\bot}}}^{\bot}$ with weight $i,$ where $A_i\neq 0.$ Then for $m\geq 4$, ${\overline{{\mathbb{C}}^{\bot}}}^{\bot}$ gives $2$-$(p^m, i, \lambda)$ designs for the following pairs:

(1) if $d'=d$ is odd,
\begin{itemize}
\item $(i, \lambda)=((p-1)(p^{m-1}-p^{s-1}), \frac{1}{2}p^{s+d-1}[(p-1)(p^{m-1}-p^{s-1})-1](p^m-1)/(p^d+1));$
\item $(i, \lambda)=( p^{m-1}(p-1)+p^{s-1}, \frac{1}{2}p^{s+d-1}[p^s(p-1)+1](p^m-p^{m-1}+p^{s-1}-1)(p^s+1)/(p^d+1));$
\item $(i, \lambda)=( (p-1)(p^{m-1}+p^{s-1}), \frac{1}{2}p^{s+d-1}[(p-1)(p^{m-1}+p^{s-1})-1](\\p^m-2p^{m-d}+1)/(p^d-1));$
\item $(i, \lambda)=( p^{m-1}(p-1)-p^{s-1}, p^{s+d-1}[p^s(p-1)-1](p^m-2p^{m-d}+1)(p^m-p^{m-1}-p^{s-1}-1)/2(p^d-1)(p^s+1));$
\item $(i, \lambda)=( p^{m-1}(p-1)\pm(-1)^{\frac{p-1}{2}}p^{s+\frac{d-1}{2}}, \frac{1}{2}p^{m-\frac{3d+1}{2}}[p^{s-\frac{d-1}{2}-1}(p-1)\pm(-1)^{\frac{p-1}{2}}][p^{m-1}(p-1)\pm(-1)^{\frac{p-1}{2}}p^{s+\frac{d-1}{2}}-1]);$   \item $(i, \lambda)=( p^{s+d-1}(p-1)(p^{s-d}+1),
    p^{s-d-1}[p^{s+d-1}(p-1)(p^{s-d}+1)-1](p^{s-d}+1)(p^{s-d}-1)/(p^{2d}-1)).$
\item $(i, \lambda)=(p^{m-1}(p-1)-p^{s+d-1},
      p^{s-d-1}[p^{s-d}(p-1)-1](p^m-p^{m-1}\\-p^{s+d-1}-1)(p^{s-d}-1)/(p^{2d}-1));$
\item $(i, \lambda)=( p^{m-1}(p-1), (p^m-p^{m-1}-1)(p^{3s-d}-p^{3s-2d}+p^{3s-2d-1}+p^{3s-3d}-p^{m-2d}+1));$

\end{itemize}

(2) if $d'=d$ is even,
\begin{itemize}
\item $(i, \lambda)=(p^{s-1}(p-1)(p^s-1), \frac{1}{2}p^{s+d-1}[p^{s-1}(p-1)(p^s-1)-1](p^m-1)/(p^d+1));$
\item $(i, \lambda)=( p^{s-1}(p^{s+1}-p^s+1), \frac{1}{2}p^{s+d-1}(p^{s+1}-p^s+1) [p^{s-1}(p^{s+1}-p^s+1)-1](p^s+1)/(p^d+1));$
\item $(i, \lambda)=(p^{s-1}(p-1)(p^s+1), p^{d-2}[p^{s-1}(p^s+1)(p-1)-1](p^m-2p^{m-d}+1)(p^s+1)(p^{s+1}-2p+2)/2(p^d-1)(p^s-1));$
\item $(i, \lambda)=(p^{s-1}(p^{s+1}-p^s-1), p^{d-2}(p^{s+1}-p^s-1)[p^{s-1}(p^{s+1}-p^s-1)-1](p^m-2p^{m-d}+1)(p^{s+1}-2p+2)/2(p^d-1)(p^s-1));$
\item $(i, \lambda)=(p^{s+\frac{d}{2}-1}(p-1)(p^{s-\frac{d}{2}}\pm1), \frac{1}{2}p^{m-\frac{3d}{2}-1}(p^{s-\frac{d}{2}}\pm1)[p^{s+\frac{d}{2}-1}(p-1)(p^{s-\frac{d}{2}}\pm1)-1]);$
\item $(i, \lambda)=(p^{s+\frac{d}{2}-1}(p^{s-\frac{d}{2}+1}-p^{s-\frac{d}{2}}\pm1), \frac{1}{2}p^{m-\frac{3d}{2}-1}(p^{s-\frac{d}{2}+1}-p^{s-\frac{d}{2}}\pm1)[p^{s+\frac{d}{2}-1}(p^{s-\frac{d}{2}+1}-p^{s-\frac{d}{2}}\pm1)-1]);$
\item $(i, \lambda)=(p^{s+d-1}(p-1)(p^{s-d}+1), p^{s-d-1}[p^{s+d-1}(p-1)(p^{s-d}+1)-1](p^{m-2d}-1)/(p^{2d}-1));$
\item $(i, \lambda)=( p^{s+d-1}(p^{s-d+1}-p^{s-d}-1), p^{s-d-1}(p^{s-d+1}-p^{s-d}-1)(p^{s-d}\\-1)[p^{s+d-1}(p^{s-d+1}-p^{s-d}-1)-1]/(p^{2d}-1));$
\item $(i, \lambda)=( p^{m-1}(p-1), (p^m-p^{m-1}-1)(p^{3s-d}-p^{3s-2d}+p^{3s-3d}-p^{m-2d}+1)).$
\end{itemize}

(3) if $d'=2d$,
\begin{itemize}
\item $(i, \lambda)=(p^{s-1}(p-1)(p^s+1), p^{s+3d-1}[p^{s-1}(p-1)(p^s+1)-1](p^m-p^{m-2d}-p^{m-3d}+p^s-p^{s-d}+1)/(p^d+1)(p^{2d}-1));$
\item $(i, \lambda)=( p^{s-1}(p^{s+1}-p^s-1), p^{s+3d-1}(p^{s+1}-p^s-1)[p^{s-1}(p^{s+1}-p^s-1)-1](p^m-p^{m-2d}-p^{m-3d}+p^s-p^{s-d}+1)/(p^d+1)(p^{2d}-1)(p^s+1));$
\item $(i, \lambda)=(p^{s+d-1}(p-1)(p^{s-d}-1), p^{s-1}(p^{s-d}-1)[p^{s+d-1}(p-1)(p^{s-d}-1)-1](p^s+p^{s-d}+p^{s-2d}+1)/(p^d+1)^2);$
\item $(i, \lambda)=( p^{s+d-1}(p^{s-d+1}-p^{s-d}+1),  p^{s-1}(p^{s-d+1}-p^{s-d}+1)[p^{s+d-1}\\(p^{s-d+1}-p^{s-d}+1)-1](p^s+p^{s-d}+p^{s-2d}+1)/(p^d+1)^2));$
\item $(i, \lambda)=(p^{s+2d-1}(p-1)(p^{s-2d}+1), p^{s-2d-1}(p^{s-d}-1)(p^{s-2d}+1)[\\p^{s+2d-1}(p-1)(p^{s-2d}+1)-1]/(p^d+1)(p^{2d}-1));$
\item $(i, \lambda)=(p^{s+2d-1}(p^{s-2d+1}-p^{s-2d}-1), p^{s-2d-1}(p^{s-2d+1}-p^{s-2d}-1)[p^{s+2d-1}(p^{s-2d+1}-p^{s-2d}-1)-1](p^{s-d}-1)/(p^d+1)(p^{2d}-1));$
\item $(i, \lambda)=( p^{m-1}(p-1), (p^m-p^{m-1}-1)(p^{3s-d}-p^{3s-2d}+p^{3s-3d}-p^{3s-4d}+p^{3s-5d}+p^{m-d}-2p^{m-2d}+p^{m-3d}-p^{m-4d}+1)).$
\end{itemize}
\end{theorem}


\section{Proofs of the main results}\label{section-4}

In this section, we will prove the results of weight distribution given in Theorem \ref{weight1} and the corresponding 2-designs described in Theorem \ref{$2-$design-1} in Section 3.  We also use Magma programs to illustrate  some examples. We begin this section by the proof of the weight distribution of the code ${\overline{{\mathbb{C}}^{\bot}}}^{\bot}.$

\begin{proof}[Proof of Theorem \ref{weight1}]
For each nonzero codeword $\mathbf{c}(a,b,c,h)=(c_0, c_1,\ldots, c_n)$ in ${\overline{{\mathbb{C}}^{\bot}}}^{\bot},$ the Hamming weight of $\mathbf{c}(a,b,c,h)$ is
\begin{equation}\label{weight formula-1}
w_H(\mathbf{c}(a,b,c,h))=n+1-T(a,b,c,h)=p^m-T(a,b,c,h),
\end{equation}
where
 \begin{eqnarray*}
 T(a,b,c,h)= |\{x\in \mathbb{F}_q: & & Tr_s(ax^{p^s+1})+Tr_m( bx^{p^l+1}+cx)+h=0,\\
& & a\in \mathbb{F}_{p^s},b,c \in \mathbb{F}_q, h\in \mathbb{F}_p\}|.
 \end{eqnarray*}
Then, it follows from the orthogonality of exponential sum and the definition of function $S(a,b,c)$ in  Eq.(\ref{S(a.b.c)}) that
\begin{eqnarray}\label{T(a,b,c,h)}
T(a,b,c,h)&=&\frac{1}{p}\sum\limits_{y \in \mathbb{F}_p}\sum\limits_{x \in \mathbb{F}_q}\zeta_p^{y[Tr_s(ax^{p^s+1})+Tr_m( bx^{p^l+1}+cx)+h]}\nonumber\\
&=&\frac{1}{p}\sum\limits_{y \in \mathbb{F}_p}\zeta_p^{yh}\sum\limits_{x \in \mathbb{F}_q}\zeta_p^{y[Tr_s(ax^{p^s+1})+Tr_m( bx^{p^l+1}+cx)]}\nonumber\\
&=&p^{m-1}+\frac{1}{p}\sum\limits_{y \in \mathbb{F}_p^*}\zeta_p^{yh}\sigma_y(S(a,b,c)).
\end{eqnarray}
Meanwhile, it follows from Lemma \ref{value distribution} that for $\varepsilon=\pm 1$ and $j \in \mathbb{F}_p^*$,
\begin{eqnarray*}
S(a,b,c)& \in &\{ \varepsilon p^s, \varepsilon \zeta_p^jp^s, \varepsilon\sqrt{p^*}p^{s+\frac{d-1}{2}}, \varepsilon\zeta_p^j\sqrt{p^*}p^{s+\frac{d-1}{2}}, -p^{s+d}, \\&& -\zeta_p^jp^{s+d},0 , p^m\}
\end{eqnarray*}
if $d'=d$ is odd, and
\begin{eqnarray*}
S(a,b,c)\in \{  \varepsilon p^s,  \varepsilon \zeta_p^jp^s,  \varepsilon p^{s+\frac{d}{2}}, \varepsilon\zeta_p^jp^{s+\frac{d}{2}}, -p^{s+d}, -\zeta_p^jp^{s+d},0 , p^m\}
\end{eqnarray*}
if $d'=d$ is even, and otherwise
\begin{eqnarray*}
S(a,b,c)\in \{-p^s, -\zeta_p^jp^s, p^{s+d}, p^{s+d}\zeta_p^j, -p^{s+2d}, -p^{s+2d}\zeta_p^j, 0, p^m\}.
\end{eqnarray*}
Then, if $d'=d$ is odd, plugging the corresponding values of $S(a,b,c)$ into Lemma \ref{Cyclotomic fields} leads to that
\begin{eqnarray*}
\sigma_y(S(a,b,c))=\left\{
\begin{array}{ll}
S(a,b,c),   & \mathrm{if}\,\   S(a,b,c)\in \{0, \varepsilon p^s,-p^{s+d},p^m\}, \\
\varepsilon p^s \zeta_p^{yj},  & \mathrm{if}\,\  S(a,b,c)=\varepsilon p^s\zeta_p^j,      \\
 \varepsilon \sqrt{p^*}p^{s+\frac{d-1}{2}}\eta'(y),   & \mathrm{if}\,\  S(a,b,c)=\varepsilon\sqrt{p^*}p^{s+\frac{d-1}{2}}, \\
\varepsilon \sqrt{p^*}p^{s+\frac{d-1}{2}}\eta'(y)\zeta_p^{yj}, &  \mathrm{if}\,\    S(a,b,c)=\varepsilon \zeta_p^{j}\sqrt{p^*}p^{s+\frac{d-1}{2}},  \\
-p^{s+d}\zeta_p^{yj}, &  \mathrm{if}\,\   S(a,b,c)=-p^{s+d}\zeta_p^{j} .
\end{array}
\right.
\end{eqnarray*}
Therefore, following from the definition of $T(a,b,c,h)$ and Eq.(\ref{T(a,b,c,h)}), together with a simple calculation,  we obtain Table \ref{8} which describes the values of $T(a,b,c,h)$ and the corresponding conditions, where $j\in \mathbb{F}_p$, which is different from others in this paper.
\begin{table}
\begin{center}
\caption{The value of $T(a,b,c,h)$ when $d'=d$ is odd ($j\in \mathbb{F}_p$)}\label{8}
\begin{tabular}{ll}
\hline\noalign{\smallskip}
Value  &  Corresponding Condition   \\
\noalign{\smallskip}
\hline\noalign{\smallskip}
$p^{m-1}+\varepsilon p^{s-1}(p-1) $  &   $ S(a,b,c)=\varepsilon p^s\zeta_p^j~\mathrm{and}~ h+j=0$    \\
$ p^{m-1}-\varepsilon p^{s-1} $  &   $ S(a,b,c)=\varepsilon p^s\zeta_p^j~\mathrm{and}~ h+j\neq0$     \\
$ 0 $  &  $ S(a,b,c)= p^m~\mathrm{and}~ h\neq0$      \\
$  p^{m-1}\pm(-1)^{\frac{p-1}{2}}p^{s+\frac{d-1}{2}} $  &  $S(a,b,c)=\varepsilon \sqrt{p^*}p^{s+\frac{d-1}{2}}\zeta_p^j~\mathrm{and}~ \eta'(h+j)=\pm\varepsilon$\\
$p^{m-1}-p^{s+d-1}(p-1)$  &  $S(a,b,c)=-p^{s+d}\zeta_p^j ~\mathrm{and}~ h+j=0$  \\
$p^{m-1}+p^{s+d-1}$ &  $ S(a,b,c)$ $=-p^{s+d}\zeta_p^j ~\mathrm{and}~ h+j\neq0 $  \\
$p^{m-1}$ &  $S(a,b,c)=0~ \mathrm{or}~ S(a,b,c)=\varepsilon\sqrt{p^*}\zeta_p^jp^{s+\frac{d-1}{2}}~ \mathrm{and}~ h+j = 0 $  \\
$p^m$ & $S(a,b,c)=p^m~ \mathrm{and}~ h = 0 $\\
\noalign{\smallskip}
\hline
\end{tabular}
\end{center}
\end{table}

Consequently, by Lemmas \ref{Gauss} and  \ref{value distribution},  Eq.(\ref{weight formula-1}) and the orthogonality of exponential sum, we get the weight distribution of  ${\overline{{\mathbb{C}}^{\bot}}}^{\bot}$ in  Table \ref{9}  when $d'=d$ is odd.
\begin{table}
\begin{center}
\caption{The weight  distribution of  ${\overline{{\mathbb{C}}^{\bot}}}^{\bot}$ when $d'=d$ is odd}\label{9}
\begin{tabular}{ll}
\hline\noalign{\smallskip}
Value  &  Multiplicity   \\
\noalign{\smallskip}
\hline\noalign{\smallskip}
$(p-1)(p^{m-1}-p^{s-1})$  &  $M_1+(p-1)M_3$ \\
$p^{m-1}(p-1)+p^{s-1}$  &  $ (p-1)M_1+(p-1)^2M_3 $    \\
$ (p-1)(p^{m-1}+p^{s-1})$  &  $M_2+(p-1)M_4 $     \\
$ p^{m-1}(p-1)-p^{s-1} $  &  $(p-1)M_2+(p-1)^2M_4$     \\
$ p^{m-1}(p-1)\pm(-1)^{\frac{p-1}{2}}p^{s+\frac{d-1}{2}}$  &  $(p-1)M_5+\frac{(p-1)^2}{2}(M_{6,1}+M_{6,-1}) $  \\
$p^{s+d-1}(p-1)(p^{s-d}+1)$ &  $ M_7+(p-1)M_8$  \\
$p^{m-1}(p-1)-p^{s+d-1}$ &  $ (p-1)M_7+(p-1)^2M_8$  \\
$p^{m-1}(p-1)$ &  $ pM_9+2M_5+(p-1)(M_{6,1}+M_{6,-1})$  \\
$p^m$ &  $ p-1 $  \\
\noalign{\smallskip}
\hline
\end{tabular}
\end{center}
\end{table}

In a similar manner, we can obtain the weight distribution of ${\overline{{\mathbb{C}}^{\bot}}}^{\bot}$ in Table \ref{10} when $d'=d$ is even and in Table \ref{11} when  $d'=2d$, respectively.
\begin{table}
\begin{center}
\caption{The weight  distribution of ${\overline{{\mathbb{C}}^{\bot}}}^{\bot}$ when $d'=d$ is even}\label{10}
\begin{tabular}{ll}
\hline\noalign{\smallskip}
Value  &  Multiplicity   \\
\noalign{\smallskip}
\hline\noalign{\smallskip}
$p^{s-1}(p-1)(p^s-1)$  &  $M_1+(p-1)M_3$ \\
$p^{s-1}(p^{s+1}-p^s+1)$  &  $ (p-1)M_1+(p-1)^2M_3 $    \\
$ p^{s-1}(p-1)(p^s+1)$  &  $M_2+(p-1)M_4 $     \\
$ p^{s-1}(p^{s+1}-p^s-1)$  &  $(p-1)M_2+(p-1)^2M_4$     \\
$ p^{s+\frac{d}{2}-1}(p^{s-\frac{d}{2}+1}-p^{s-\frac{d}{2}}\pm1)$  &  $(p-1)M_{5,\pm1}+(p-1)^2M_{6,\pm1} $  \\
$p^{s+\frac{d}{2}-1}(p-1)(p^{s-\frac{d}{2}}\pm1)$ &  $M_{5,\mp1}+(p-1)M_{6,\mp1}$  \\
$p^{s+d-1}(p-1)(p^{s-d}+1)$ &  $M_7+(p-1)M_8$  \\
$p^{s+d-1}(p^{s-d+1}-p^{s-d}-1)$ &  $(p-1)M_7+(p-1)^2M_8 $  \\
$p^{m-1}(p-1)$ &  $pM_9 $  \\
$p^m$ &  $p-1 $  \\
\noalign{\smallskip}
\hline
\end{tabular}
\end{center}
\end{table}

\begin{table}
\begin{center}
\caption{The weight distribution of ${\overline{{\mathbb{C}}^{\bot}}}^{\bot}$ when $d'=2d$}\label{11}
\begin{tabular}{ll}
\hline\noalign{\smallskip}
Value  &  Multiplicity   \\
\noalign{\smallskip}
\hline\noalign{\smallskip}
$p^{s-1}(p-1)(p^s+1)$  &  $M_1+(p-1)M_2$ \\
$p^{s-1}(p^{s+1}-p^s-1)$  &  $ (p-1)M_1+(p-1)^2M_2 $    \\
$p^{s+d-1}(p-1)(p^{s-d}-1)$  &  $M_3+(p-1)M_4 $     \\
$ p^{s+d-1}(p^{s-d+1}-p^{s-d}+1)$  &  $(p-1)M_3+(p-1)^2M_4$     \\
$ p^{s+2d-1}(p-1)(p^{s-2d}+1) $  &  $M_5+(p-1)M_6$     \\
$ p^{s+2d-1}(p^{s-2d+1}-p^{s-2d}-1)$  &  $(p-1)M_5+(p-1)^2M_6 $  \\
$p^{m-1}(p-1)$ &  $pM_7$  \\
$p^m$ &  $p-1 $  \\
\noalign{\smallskip}
\hline
\end{tabular}
\end{center}
\end{table}

Thus we completed the proof of Theorem \ref{weight1} by applying the corresponding values of $M_i$ or $M_{i,j}$ in Tables \ref{1}-\ref{3} into Tables \ref{9}-\ref{11}, respectively.
\end{proof}

Below, we will derive infinite families of $2$-designs from ${\overline{{\mathbb{C}}^{\bot}}}^{\bot}$. To this end, we first prove that the code  is affine-invariant.

\begin{lemma}\label{affine invariant}
The extended code $\overline{{\mathbb{C}}^{\bot}}$ is affine-invariant.
\end{lemma}

\begin{proof}
It is easy to know that the defining set $T$ of the cyclic code $\mathbb{C}^{\bot}$ is $T =C_1 \cup C_{p^l+1} \cup C_{p^s+1}$. Since $0 \not \in T$, the defining set $\overline{T}$ of $\overline{{\mathbb{C}}^{\bot}}$ is given by $\overline{T} = C_1 \cup C_{p^l+1} \cup C_{p^s+1} \cup \{0\}$.
Let $u \in \overline{T} $ and $r \in \mathcal{P}$ with $r \preceq u $. In order to apply Lemma \ref{Kasami-Lin-Peterson}, we need to prove $r \in \overline{T}$.

 If $r=0,$ it is easy to verify that $r\in \overline{T}$. We consider now the case $r>0$. If $u \in  C_1$, then the Hamming weight $wt(u) = 1.$ Note the fact that  $r \preceq u$ implies  $wt(r) = 1$ and thus $r \in C_1 \subset  \overline{T}.$  If $u  \in C_{p^l+1}\cup C_{p^s+1}$, then the Hamming weight $wt(u) = 2.$ With the assumption that  $r \preceq u$, either $wt(r) = 1$ or $r = u.$ In both cases, $r \in  \overline{T}.$ The desired conclusion then follows
from Lemma \ref{Kasami-Lin-Peterson}.
Thus the proof is completed.
\end{proof}

Now we are ready to prove Theorems \ref{$2-$design-1} and  \ref{parameter-1}.

\begin{proof}[Proof of Theorem \ref{$2-$design-1}] By Lemma \ref{Kasami-Lin-Peterson}(2) and Lemma \ref{affine invariant}, we know ${\overline{{\mathbb{C}}^{\bot}}}^{\bot}$ is affine-invariant. Thus we  complete the proof by Theorem \ref{2-design}.
\end{proof}

\begin{proof}[Proof of Theorem \ref{parameter-1}]
One can easily verify that the minimum weight and the maximum weight of the code ${\overline{{\mathbb{C}}^{\bot}}}^{\bot}$ satisfy the condition of
 Theorem \ref{design parameter}. Hence,
 the number of supports of all codewords with weight $i\neq 0$ in the code ${\overline{{\mathbb{C}}^{\bot}}}^{\bot}$ is equal to $A_i/(p-1)$ for each $i,$ where $A_i$ is described  in Tables \ref{4}-\ref{6}.  Then the desired conclusions follow from Theorem \ref{$2-$design-1} and  Eq.(\ref{condition}).
\end{proof}

We conclude this section by presenting some examples to illustrate the validity of the results.

\begin{example}\label{example1}
If $(p, s, l)=(3, 2, 1)$, then the code ${\overline{{\mathbb{C}}^{\bot}}}^{\bot}$ has parameters $[81,11,\\45]$ and weight enumerator $1+6840z^{45}+24300z^{48}+27216z^{51}+49920z^{54}+48600z^{57}+13608z^{60}+6480z^{63}+180z^{72}+2z^{81},$ which  agrees with  the results given in Theorem \ref{weight1}.
\end{example}


\begin{example}\label{example2}
If $(p, s, l)=(3, 3, 2)$, then the code ${\overline{{\mathbb{C}}^{\bot}}}^{\bot}$ has parameters $[729,16,\\459]$ and weight enumerator $1+1710072z^{459}+5572476z^{468}+6937164z^{477}+12562368z^{486}+11144952z^{495}+3468582z^{504}+1592136z^{513}+58968z^{540}+2z^{729},$ which agrees with  the results given in Theorem \ref{weight1}.
\end{example}

\begin{example}\label{example3}
If $(p, s, l)=(3, 3, 1)$, then the code ${\overline{{\mathbb{C}}^{\bot}}}^{\bot}$ has parameters $[729,16,\\405]$ and weight enumerator $1+3276z^{405}+442260z^{432}+20470320z^{477}+\\11009544z^{486}+10235160z^{504}+884520z^{513}+1638z^{648}+2z^{729},$ which agrees with  the results given in Theorem \ref{weight1}.
\end{example}

\begin{example}\label{example-4}
If $(p,s,l)=(3,2,1)$, then the code ${\overline{{\mathbb{C}}^{\bot}}}^{\bot}$ has parameters $[81, 11, 45]$ and the weight distribution is given in Example \ref{example1}. It gives $2$-$(81, i, \lambda)$ designs with the following pairs $(i, \lambda):$
$$(45, 1045), (48, 4230), (51, 5355), (54, 11024),$$ $$(57, 11970), (60, 3717), (63, 1953), (72, 71),$$
which confirms the results given in Theorem \ref{parameter-1}.
\end{example}

\section{Concluding remarks}\label{section-5}
In this paper, we first determined the weight distributions of a class of linear codes derived from the duals of extended non-binary Kasami cyclic codes. Using the properties of affine-invariant codes, we then found that ${\overline{{\mathbb{C}}^{\bot}}}^{\bot}$ could give infinite families of  $2$-designs and explicitly determined their parameters. 

\section*{Appendix I}
\begin{table}
\begin{center}
\caption{The value distribution of $S(a,b,c)$ when $d'=d$ is odd}\label{1}
\begin{tabular}{ll}
\hline\noalign{\smallskip}
Value  &  Multiplicity   \\
\noalign{\smallskip}
\hline\noalign{\smallskip}
$p^s$  &  $M_1= \frac{1}{2}p^{s+d-1}(p^s+1)(p^s+p-1)(p^m-1)/(p^d+1)$ \\
$-p^s$  &  $ M_2=\frac{1}{2}p^{s+d-1}(p^s-1)(p^s-p+1)(p^m-2p^{m-d}+1)/(p^d-1) $    \\
$ \zeta^j_pp^s$  &  $M_3=\frac{1}{2}p^{s+d-1}(p^m-1)^2/(p^d+1) $     \\
$ -\zeta^j_pp^s $  &  $M_4=\frac{1}{2}p^{s+d-1}(p^m-2p^{m-d}+1)(p^m-1)/(p^d-1)$     \\
$  \varepsilon\sqrt{p^*}p^{s+\frac{d-1}{2}}  $  &  $M_5=\frac{1}{2}p^{3s-2d-1}(p^m-1) $     \\
$ \varepsilon\sqrt{p^*}p^{s+\frac{d-1}{2}}\zeta^j_p$  &  $M_{6,\varepsilon}=\frac{1}{2}p^{m-\frac{3d+1}{2}}(p^{s-\frac{d+1}{2}}+\varepsilon\eta'(-j))(p^m-1) $  \\
$-p^{s+d}$ &  $ M_7=p^{s-d-1}(p^{s-d}-1)(p^{s-d}-p+1)(p^m-1)/(p^{2d}-1)$  \\
$-p^{s+d}\zeta^j_p$ &  $ M_8= p^{s-d-1}(p^{m-2d}-1)(p^m-1)/(p^{2d}-1)$  \\
$0$ &  $ M_9=(p^{3s-d}-p^{3s-2d}+p^{3s-3d}-p^{m-2d}+1)(p^m-1) $  \\
$p^m$ &  $ 1 $  \\
\noalign{\smallskip}
\hline
\end{tabular}
\end{center}
\end{table}

\begin{table}
\begin{center}
\caption{The value distribution of $S(a,b,c)$ when $d'=d$ is even}\label{2}
\begin{tabular}{ll}
\hline\noalign{\smallskip}
Value  &  Multiplicity   \\
\noalign{\smallskip}
\hline\noalign{\smallskip}
$p^s$  &  $M_1= \frac{1}{2}p^{s+d-1}(p^s+1)(p^s+p-1)(p^m-1)/(p^d+1)$ \\
$-p^s$  &  $ M_2=\frac{1}{2}p^{s+d-1}(p^s+1)(p^s-p+1)(p^m-2p^{m-d}+1)/(p^d-1) $    \\
$ \zeta^j_pp^s$  &  $M_3=\frac{1}{2}p^{s+d-1}(p^m-1)^2/(p^d+1) $     \\
$ -\zeta^j_pp^s $  &  $M_4=\frac{1}{2}p^{s+d-1}(p^m-2p^{m-d}+1)(p^m-1)/(p^d-1)$     \\
$  \varepsilon p^{s+\frac{d}{2}}  $  &  $M_{5,\varepsilon}=\frac{1}{2}p^{m-\frac{3d}{2}-1}(p^{s-\frac{d}{2}}+\varepsilon(p-1))(p^m-1) $     \\
$ \varepsilon p^{s+\frac{d}{2}}\zeta^j_p$  &  $M_{6,\varepsilon}=\frac{1}{2}p^{m-\frac{3d}{2}-1}(p^{s-\frac{d}{2}}-\varepsilon)(p^m-1) $  \\
$-p^{s+d}$ &  $ M_7=p^{s-d-1}(p^{s-d}-1)(p^{s-d}-p+1)(p^m-1)/(p^{2d}-1)$  \\
$-p^{s+d}\zeta^j_p$ &  $ M_8= p^{s-d-1}(p^{s-d}-1)(p^{s-d}+1)(p^m-1)/(p^{2d}-1)$  \\
$0$ &  $ M_9=(p^{3s-d}-p^{3s-2d}+p^{3s-3d}-p^{m-2d}+1)(p^m-1)$  \\
$p^m$ &  $ 1 $  \\
\noalign{\smallskip}
\hline
\end{tabular}
\end{center}
\end{table}

\begin{table}
\begin{center}
\caption{The value distribution of $S(a,b,c)$ when $d'=2d$}\label{3}
\begin{tabular}{ll}
\hline\noalign{\smallskip}
Value  &  Multiplicity   \\
\noalign{\smallskip}
\hline\noalign{\smallskip}
$-p^s$  &  $ M_1=\frac{p^{s+3d-1}(p^s-1)(p^s-p+1)(p^m-p^{m-2d}-p^{m-3d}+p^s-p^{s-d}+1)}{(p^d+1)(p^{2d}-1)} $    \\
$ -\zeta^j_pp^s $  &  $M_2=\frac{p^{s+3d-1}(p^m-p^{m-2d}-p^{m-3d}+p^s-p^{s-d}+1)(p^m-1)}{(p^d+1)(p^{2d}-1)}$     \\
$  p^{s+d}  $  &  $M_3=\frac{p^{s-1}(p^{s-d}+p-1)(p^s+p^{s-d}+p^{s-2d}+1)(p^m-1)}{(p^d+1)^2}$      \\
$  p^{s+d}\zeta^j_p$  &  $M_4=\frac{p^{s-1}(p^{s-d}-1)(p^s+p^{s-d}+p^{s-2d}+1)(p^m-1)}{(p^d+1)^2} $  \\
$-p^{s+2d}$ &  $ M_5=\frac{p^{s-2d-1}(p^{s-d}-1)(p^{s-2d}-p+1)(p^m-1)}{(p^d+1)(p^{2d}-1)}$  \\
$-p^{s+2d}\zeta^j_p$ &  $ M_6=\frac{p^{s-2d-1}(p^{s-d}-1)(p^{s-2d}+1)(p^m-1)}{(p^d+1)(p^{2d}-1)}$  \\
$0$ &  $ M_7=(p^{3s-d}-p^{3s-2d}+p^{3s-3d}-p^{3s-4d}+p^{3s-5d}+p^{m-d}-$\\&$2p^{m-2d}+p^{m-3d}-p^{m-4d}+1)(p^m-1) $  \\
$p^m$ &  $ 1 $  \\
\noalign{\smallskip}
\hline
\end{tabular}
\end{center}
\end{table}


\section*{ACKNOWLEDGMENTS}
The authors would like to thank the anonymous referees for their helpful comments and suggestions, which have greatly improved the presentation and quality of this paper. The research of X. Du was supported by NSFC No. 61772022. The research of C. Fan was supported by NSFC No. 11971395.



\end{document}